\documentclass[onefignum,onetabnum]{siamart190516}
\usepackage{epsfig}

\title{On complex eigenvalues of a real nonsymmetric matrix}
\author{Andy Wathen\thanks{Computational Mathematics Group, Rutherford Appleton Laboratory, UK (andrew.wathen@stfc.ac.uk)}}
\ifpdf
\hypersetup{
  pdftitle={Complex eigenvalues of nonsymmetric real matrices},
  pdfauthor={A. J. Wathen}
}
\fi

\usepackage{graphicx}
\usepackage{amsmath}
\usepackage{amsfonts}
\usepackage{MnSymbol}

\newfont{\sym}{cmr12}

\newtheorem{thm}{Theorem}[section]
\newtheorem{lem}[thm]{Lemma}
\newtheorem{cor}[thm]{Corollary}
\newtheorem{prop}[thm]{Proposition}

\newtheorem{exmpl}[thm]{{\it Example}}
\newtheorem{rem}[thm]{{\it Remark}}

\def\smallskip{\vskip 5pt plus1pt minus1pt}
\def\medskip{\vskip 9pt plus2pt minus2pt}
\def\bigskip{\vskip 18pt plus4pt minus4pt}

\def\iddots{\reflectbox{$\ddots$}}
\def\half{{\frac{1}{2}}}

\def\R{\hbox{{\msbm \char "52}}}

\def\R{{\mathbb{R}}}
\def\C{{\mathbb{C}}}

\def\sqr#1#2{{\vcenter{\vbox{\hrule height.#2pt
        \hbox{\vrule width.#2pt height#1pt \kern#1pt \vrule width.#2pt}
             \hrule height.#2pt}}}}

\def\minres{{\large {\sc minres}}}
\def\gmres{{\large {\sc gmres}}}

\def\rotatechartwo#1{\reflectbox{#1}}
\def\antiJ{\rotatechartwo{\sf J}}
\def\symJ{\,\antiJ\!\!\!\!\!\sf J}
\def\JJ{{\symJ\,}}
\def\B{{\sf B}}
\def\T{{\sf T}}
\def\W{{\sf W}}
\def\Y{{\sf Y}}
\def\V{{\sf V}}
\def\A{{\sf A}}
\def\I{{\sf I}}
\def\U{{\sf U}}
\def\S{{\sf S}}
\def\J{{\sf J}}
\def\LLambda{{\sf \Lambda}}
\def\Q{{\sf Q}}
\def\H{{\sf H}}

\begin{document}

\maketitle
\medskip

\begin{abstract}
  We consider real non-symmetric matrices and their factorisation as a
  product of real symmetric matrices. The number of
  complex eigenvalues of the original matrix reveals restrictions on
  such factorisations as we shall prove.
\end{abstract}



\section*{Introduction}

We all know that real square nonsymmetric matrices can have complex eigenvalues, but
how many? Less well-known seems to be the fact that every real square matrix can be expressed
as the product of two real {\em symmetric} matrices; such a factorisation can often be done in many
different ways. As opposed to real nonsymmetric matrices, real symmetric matrices are always
diagonalisable and have real eigenvalues
(this is sometimes called the {\em Spectral Theorem}) and
thus each real symmetric matrix has its own {\em inertia}: the triple $(p,n,z)$ where
$p$ is the number of positive
eigenvalues, $n$ is the number of negative eigenvalues and $z$ is the number
of zero eigenvalues.

In this short article we consider an invertible nonsymmetric matrix $\B\in\R^{m\times m}$ when
we express it as
$\B=\T\W$ with $\T,\W\in\R^{m\times m}$ being real symmetric. For clarity we employ the
{\em Parlett convention}
\cite[page 6]{Parlett}: symmetric characters refer to symmetric matrices.
So $\T=\T^T$ and $\W=\W^T$, but $\B\ne \B^T$ are implied by this convention.

Here we prove bounds relating the number of complex eigenvalues of
a real invertible matrix $\B$ and the inertia of real symmetric matrices $\T,\W$
which are such that $\B=\T\W$. Note that because $\B$ is invertible, then
$\T^{-1}, \W^{-1}$ must exist. Further because
$\lambda>0 \Leftrightarrow \frac{1}{\lambda}>0$, we have that the inertia of
$\T^{-1}$ is the same as the inertia of $\T$
and the inertia of $\W^{-1}$ is the same as the inertia of $\W$.
We follow up with further comment on such factorisation and give several examples.
Some of the results we present were known to Ostrowski \cite{Ostrowski} (see also
\cite{TauskyTodd}).

\section{Eigenvalues, Inertia and Bounds}

As soon as we learn about real polynomials, we encounter the
possibility of complex roots. Applied to the Characteristic Polynomial
we quickly discover that general (nonsymmetric) real matrices can
have complex eigenvalues; real {\em symmetric} matrices on the other hand
can only have real eigenvalues. In all cases the roots of a polynomial
are continuously dependent on the coefficients, hence the eigenvalues
of a matrix are continuously dependent on the matrix entries.

Suppose $\B = \T\W$, and note that $\T^{-1},\W^{-1}$ must also be real symmetric.
We then have that

\begin{tabular}{ccc}
  $\lambda $ is an eigenvalue of $\B$ &$\iff$& ${\B-\lambda \I}$ is singular\\
  &$\iff$&
  $\T\W - \lambda \I$ is singular\\
  &$\iff$&
  $\W-\lambda \T^{-1}$ is singular.
\end{tabular}

\noindent Now $\V(\theta) = \theta \W + (1-\theta) \T^{-1}$ is a real symmetric matrix for all
$\theta \in\R$, hence  $\V(\theta)$ has real eigenvalues that depend continuously
on $\theta$.

\begin{lem}
  If $\B=\T\W$ with $\T$ and $\W$ having different inertia
then $\B$ has at least one real negative eigenvalue.
\end{lem}
\begin{proof}
  Since $\T$, thus $\T^{-1} = \V(0)$ and $\W=\V(1)$ have different inertia, at least one
  eigenvalue of $\V(\theta)$ must change sign as $\theta$ varies from $0$ to $1$.
  Hence, by the Intemediate Value Theorem, there must exist $\widehat{\theta}\in(0,1)$
  such that $\V(\widehat{\theta}\,)$ has a zero eigenvalue and therefore is singular. Thus 
  $$ \W + \frac{(1-\widehat{\theta}\,)}{\widehat{\theta}}\, \T^{-1}\quad{\mbox{and so also}}\quad
  \T\W+ \frac{(1-\widehat{\theta}\,)}{\widehat{\theta}}\, \I$$ are singular, hence
  $$\frac{(\widehat{\theta}-1)}{\widehat{\theta}} <0$$ is an eigenvalue of $\T\W = \B$
  and the result follows.
\end{proof}

With a little more subtlety we can make a more precise quantitative statement. In all that
follows we suppose that $\T$ (and so $\T^{-1}$ also) has inertia $(p,n,0)$ and that $\W$
(and so $\W^{-1}$ also) has inertia $(\widehat{p},\widehat{n},0)$.

\begin{cor}\label{negeigs}
  $\B=\T\W$ has at least $|p-\widehat{p}\,| = |n-\widehat{n}\,|$ real negative
  eigenvalues (counting multiplicities).
\end{cor}
\begin{proof}
  The general principle of {\em Avoidance of Crossing} due to Lax \cite{Lax} indicates
  why there are likely to be at least $|p-\widehat{p}\,| = |n-\widehat{n}\,|$ values
  of $\widehat{\theta}\in(0,1)$ for which $\V({\widehat{\theta}}\,)$ is singular, though this
  is not quite sufficient to provide a general proof. A proof simply requires one to
  clarify mutiplicities if several eigenvalue trajectories of $\V(\theta)$
change sign at the same value
  $\widehat{\theta}$. 
  Note that
  for any such $\widehat{\theta}$, because of symmetry, the kernel of $\V({\widehat{\theta}}\,)$
  must possess an orthonormal basis ${\cal{U}} = \{ u_1,\ldots , u_r\}$ where $r$  is the nullity
  of $\V({\widehat{\theta}}\,)$. Thus
  $$\left({\widehat{\theta}}\; \W + (1-{\widehat{\theta}}\,)\, \T^{-1} \right) u_i =0 ,$$
  equivalently
  $$\left({\widehat{\theta}}\; \T\W + (1-{\widehat{\theta}}\,)\, \I \right) u_i =0 $$
  for $i=1,\ldots,r$. Dividing by ${\widehat{\theta}}$, it is seen that
  the vectors $u_1 , \ldots ,u_r$ are necessarily linearly
  independent right eigenvectors of $\B=\T\W$ corresponding to the eigenvalue
  $\widehat{\lambda} = \frac{(\widehat{\theta}-1)}{\widehat{\theta}} <0$. Thus the kernel of
  $\B-\widehat{\lambda}\, \I$ must be of dimension at least $r$ when $r$ eigenvalue trajectories
  of $V(\theta)$ change sign at the same value $\theta={\widehat{\theta}}$.
\end{proof}

By applying similar argument to the above to the different parameterised real
symmetric matrix $\U(\phi) = \phi \W + (1-\phi) (-\T^{-1})$, noting that
$\U(1)=\W$ and that $\U(0) = -\T^{-1}$ which has inertia $(n,p,0)$ we have

\begin{cor}\label{poseigs}
  The real matrix $\B=\T\W$ has at least $|n-\widehat{p}\,| = |p-\widehat{n}\,|$ real and
  positive eigenvalues (counting multiplicities).
\end{cor}
\begin{proof}
  Unless the inertia of $\W$ and $-\T$ are the same, there exist values $\widehat{\phi}\in(0,1)$
  such that $\U(\widehat{\phi}\,)$ is singular. For any such $\widehat{\phi}$, as in the
  proof of the lemma above we have that
  $\frac{1-{\widehat{\phi}}\,}{\widehat{\phi}}$ is a real, positive eigenvalue of $\B=\T\W$.
  The same quantitative argument as in the corollary above gives the result.
\end{proof}

We now come to our main result. Because the complex numbers include the real numbers, we
use `non-real' to refer to complex numbers with non-vanishing imaginary parts. As
above, the inertia of $\T$ is $(p,n,0)$.

\begin{prop}\label{mainprop}
  Suppose $\B\in\R^{m\times m}$ is invertible and $\B=\T\W$ with
  $\T=\T^T\in\R^{m\times m}$,$\W=\W^T\in\R^{m\times m}$.
  If $\B$ has $m-s$ non-real eigenvalues
  (thus $\frac{1}{2}m-\frac{1}{2}s$ complex conjugate pairs) then
  $$\frac{1}{2}m-\frac{1}{2}s \leq p \leq \frac{1}{2}m+\frac{1}{2}s .$$
\end{prop}
\begin{proof}
%
  We prove by contradiction: if $p<\frac{1}{2}m-\frac{1}{2}s$ then $n>\frac{1}{2}m+\frac{1}{2}s$
  because $p+n=m$. Likewise if $p>\frac{1}{2}m+\frac{1}{2}s$ then $n<\frac{1}{2}m-\frac{1}{2}s$.
  Thus, whatever the value of $\widehat{p}$ we must have that
  \begin{equation}\label{numeigs}
|p-\widehat{p}\,|+|n-\widehat{p}\,|>\frac{1}{2}m+\frac{1}{2}s-\Bigl(\frac{1}{2}m-\frac{1}{2}s\Bigr)=s.
  \end{equation}
  By the corollaries \ref{negeigs} and \ref{poseigs} above, there are at least
  $|p-\widehat{p}\,|$ real negative eigenvalues and at least $|n-\widehat{p}\,|$ real
  positive eigenvalues of $\B$, thus the inequality (\ref{numeigs}) implies that there must be
  greater than $s$ real eigenvalues of $\B$. Because it has $m$ eigenvalues in total
  (from the fundamental theorem of algebra), $\B$ must have less than $m-s$ non-real eigenvalues.
  This completes the proof by contradiction.
\end{proof}

A similar result applies to $\W$:

\begin{cor}\label{hatW}
  If $\B=\T \W$ as above has $m-s$ non-real eigenvalues then
  $$\frac{1}{2}m-\frac{1}{2}s \leq \widehat{p} \leq \frac{1}{2}m+\frac{1}{2}s .$$ 
\end{cor}
\begin{proof}
  Note $\B^T =\W\T$ because $\T,\W$ are symmetric. The result follows from
  Proposition \ref{mainprop} by reversing the roles of $\T$ and $\W$ and noting that
  $\B^T$ and $\B$ are similar:
  $$ \T^{-1} \B \T = \T^{-1} \T\W \T = \W\T = \B^T ,$$
  therefore their eigenvalues must be the same.
\end{proof}

\begin{cor}\label{maincor}
  If $\B=\T\W$ as above has all $m$ eigenvalues non-real then
  $${\mbox{inertia of }}\T = (\frac{1}{2}m,\frac{1}{2}m,0) = {\mbox{inertia of }}\W$$
  (and, of course, $m$ has to be even).
\end{cor}
\begin{proof}
  Take $s=0$ in Proposition \ref{mainprop} and Corollary \ref{hatW} above.
\end{proof}
\begin{rem}
  Taking $s=m$ in Proposition \ref{mainprop} and Corollary \ref{hatW} (meaning all
  eigenvalues of $\B\!$ are real) gives the unhelpful (but correct)
  inequalities $0\leq p, \widehat{p} \leq m$. These are evidently as tight as possible in
  this situation since, for example,  $\I=\A \A^{-1}$ for any invertible real symmetric
  matrix, $\A$, regardless of it's inertia. Of course, $\I$ has all real eigenvalues!
\end{rem}

\section{Factorisation}
Non-uniqueness of the real symmetric factors $\T,\W$ in the factorisation $\B=\T\W$
is evident from the preceeding remark, however proof of existence generally requires the
{\em{Jordan Canonical Form}} (see for example \cite[Chapter 3]{Horn_Johnson}): for any
$\B\in\R^{m\times m}$ there exists an invertible matrix
$\S$ such that $$\B=\S \J \S^{-1}$$ where $\J$ is a block diagonal matrix of Jordan blocks
each of the form
$$\widehat{\J}= \left[\begin{array}{ccccc} \lambda & 1&&&\\
    &\lambda&1  &&\\
    &&\ddots&\ddots&\\
    &&&\lambda&1\\
    &&&&\lambda
  \end{array}\right]
\in\C^{\ell\times \ell}
$$
where all entries not specified are zero. Here $\lambda$ is the same eigenvalue
corresponding to which there is only one
  eigenvector. For a simple eigenvalue, we would have $\ell=1$.
  The same eigenvalue can appear in the different Jordan blocks that comprise $\J$; the number
  of such blocks with the same eigenvalue is the geometric multiplicity of that eigenvalue
  whereas the algebraic multiplicity of that eigenvalue is the sum of the dimensions of such blocks.
  The important observation in our context is that
  $${\widehat{\J}} = \left[\begin{array}{ccccc}
    &&&&1\\
    &&&1&\\
    &&\iddots&&\\
    &1&  &&\\
      1 & &&&
    \end{array}\right]
  \left[\begin{array}{ccccc}
    &&&&\lambda\\
    &&&\lambda&1\\
    &&\iddots&\iddots&\\
    &\lambda&1  &&\\
      \lambda & 1&&&
    \end{array}\right] = {\widehat{\Y}}\, {\widehat{\JJ}}
  $$
  where, as implied by the notation, the matrix $\widehat{\Y}$
  and the
  matrix ${\widehat{\JJ}}$ are symmetric. Whether ${\widehat{\JJ}}$ is real or not depends on
  whether $\lambda$ is real or non-real (complex). Since this holds for each Jordan block
  (whether they are of dimension $1$ or greater), one can construct block diagonal matrices
  $\Y\in\R^{m\times m},\JJ\in\C^{m\times m}$ from these blocks such
  that $\J=\Y \JJ$ with $\Y$ and $\JJ$ being symmetric.
  Thus
  $$\B=\S\J\S^{-1}=\S\Y\JJ\,\S^{-1}=\underbrace{\S\Y\S^T}_{\T} \underbrace{\S^{-T}\JJ\,\S^{-1}}_{\W}$$
  where $\T$ and $\W$ are symmetric;
  $\S^{-T}$ means the inverse of the transpose or the transpose of the inverse which
  is the same matrix.
  Here $\T$ is necessarily real symmetric
if $\S$ can be chosen to be real;
  it is then {\em congruent} over $\R$
  to $\Y$, the inertia of which is obtained by summing the inertia of each $\widehat{\Y}$.
  It is readily checked that the inertia of $\widehat{\Y}\in\R^{\ell\times \ell}$ is
  $(\lfloor{\frac{\ell +1}{2}}\rfloor,\lfloor{\frac{\ell}{2}}\rfloor,0)$.
  Sylvester's Law of Inertia
  (see for example \cite[page 403]{Golub_vanLoan} or \cite[Section 6.3]{Strang}) guarantees
  that $\Y$ and $\T$ have the same inertia.

  If all eigenvalues are real, then $\widehat{\JJ}\,$, $\JJ$ and $\W$ are also real symmetric
  since for this situation one can take $\S\in\R^{m\times m}$ (see \cite[Theorem 3.1.11]{Horn_Johnson}).
  In this case the inertia of $\widehat{\JJ}$ is the same as the inertia of $\widehat{\Y}$
  if $\lambda>0$ and
  $(\lfloor{\frac{\ell}{2}}\rfloor,\lfloor{\frac{\ell+1}{2}}\rfloor,0)$ if $\lambda<0$; the
  inertia of $\JJ$ can likewise be obtained by summing over the blocks.

  If there are non-real eigenvalues, then one can employ the {\em real Jordan form} to remain
  with only real values. The basic idea here is that
  since a real matrix can only have complex
  conjugate pairs of eigenvalues, $a\pm i b$, and complex conjugate eigenvalues must have the same
  geometric multiplicity (i.e. the same number and size of Jordan blocks), then the
  real $2\times 2$ diagonal block
  and it's factorisation as a product of real symmetric matrices  
$$\left[\begin{array}{cc} a& -b\\b&a \end{array}\right] =
\left[\begin{array}{cc} 0& 1\\1&0 \end{array}\right]
\left[\begin{array}{cc} b& a\\a&-b \end{array}\right]$$
can be employed. For
a thorough description of the real Jordan form, see \cite[pages 151--153]{Horn_Johnson}, but
also see Example 3 below.
The matrix $\S$ can always be taken to be real with the real Jordan form. 

\begin{exmpl}
  If $\B$ is diagonalisable over $\R$, that is it is diagonalisable and has
  all real eigenvalues. In this situation
all Jordan blocks are of dimension $1\times 1$ and so
  $$\B = \S\LLambda \S^{-1}$$ for some invertible
  matrix $\S\in\R^{m\times m}$ and a diagonal matrix of real eigenvalues
  $$\LLambda={\mbox{diag}}(\lambda_1, \lambda_2.,\ldots,\lambda_m) .$$
  Thus $$\B = \underbrace{\S\S^T}_{\T} \underbrace{\S^{-T} \LLambda\S^{-1}}_{\W}$$
  is one possible factorisation as the product of real symmetric matrices. Note
  here that $\T$ is necessarily positive definite since for any non-zero $x\in\R^m$
  $$x^T \T x = x^T \S \S^T x = (\S^T x)^T (\S^T x)>0,$$ and that the inertia of $\W$
  is the same as the inertia of $\LLambda$ because of Sylvester's Law of Inertia.
  Importantly, because $\T$ is symmetric and positive definite it is diagonalisable
  with orthonormal eigenvectors and positive real eigenvalues:
  $\T= \Q \H \Q^T$ where $\Q$ is an orthogonal matrix and $\H$ is a real
  diagonal matrix with entries $h_{i,i}>0$ which are the eigenvalues of $\T$.
  In this situation, an invertible symmetric and
  positive definite square root exists:
  $$\T^{\frac{1}{2}} = \Q \H^{\frac{1}{2}} \Q^T$$ where $\H^{\frac{1}{2}}$ is the diagonal matrix
  with real diagonal entries $+\sqrt{h_{i,i}}>0$. The similarity transformation
  $$\T^{-\frac{1}{2}} \B \T^{\frac{1}{2}} = \T^{-\frac{1}{2}} \T\W \T^{\frac{1}{2}} = 
  \T^{\frac{1}{2}} \W \T^{\frac{1}{2}}
  $$
  shows that $\B$ is similar to the real symmetric matrix
  $\T^{\frac{1}{2}} \W \T^{\frac{1}{2}}$
  and so necessarily has the same (real) eigenvalues.
  In fact $\B$ is {\em self-adjoint} in the inner product defined by
  $\langle x,y\rangle_{\T^{-1}}:= y^T \T^{-1} x$.
  We can easily see the non-uniqueness in this situation by noticing that there are many
  ways to express a real diagonal matrix as the product of real diagonal matrices:
  $\LLambda = \LLambda_1 \LLambda_2 $ ; for example
  \begin{eqnarray*}
    \LLambda_1 &=& {\mbox{diag}}(-\lambda_1,\pm\sqrt{|\lambda_2|},\ldots, 2),\\
    \LLambda_2&=&{\mbox{diag}}(-1,\sqrt{|\lambda_2|},\ldots,\frac{1}{2}\lambda_m)\end{eqnarray*}
  (where $+ or -$ is used to get the correct sign of $\lambda_2$). Clearly 
  $$\B=\underbrace{\S\LLambda_1\S^T}_{\T} \underbrace{\S^{-T} \LLambda_2\S^{-1}}_{\W}.$$
  It is thus seen that $\T,\W$ can have any inertia in this situation.
\end{exmpl}

\begin{exmpl}
  If $\B\in\R^{m\times m}$ has only one distinct eigenvalue (it must be real) and only
  one eigenvector, then it has a single Jordan block as above with $\ell = m$:
  $$\B=\S \widehat{\J} \S^{-1}=
  \underbrace{\S\widehat{\Y}\S^T}_{\T} \underbrace{\S^{-T}\,\widehat{\JJ}\,\S^{-1}}_{\W} $$
  for some real invertible matrix $\S$. 
  Note that if $m$ is even then here is an example where
  $$\mbox{inertia of }\T = (\frac{m}{2},\frac{m}{2},0) = \mbox{inertia of }\W$$
  but all eigenvalues are real; the converse of Corollary \ref{maincor} is certainly false!
\end{exmpl}

\begin{exmpl}
  If $\B\in\R^{6\times 6}$ has a complex conjugate pair of eigenvalues $a\pm i b$ which
  are double eigenvalues but with only one eigenvector each as well as a further
  complex conjugate pair of simple eigenvalues $c\pm i d$, then by invoking the real Jordan form,
  there is an invertible matrix $\S$ with
\begin{eqnarray*}
    \B&=&\S
    \left[\begin{array}{cccccc}
c&-d&0&0&0&0\\
d&c&0&0&0&0\\
        0&0&a&-b&1&0\\
        0&0&b&a&0&1\\
        0&0&0&0&a&-b\\
        0&0&0&0&b&a        
      \end{array}\right]
    \S^{-1}\\
    &=&\S\left[\begin{array}{cccccc}
 0&1&0&0&0&0\\
        1&0&0&0&0&0\\
        0&0&0&0&0&1\\
        0&0&0&0&1&0\\
        0&0&0&1&0&0\\
        0&0&1&0&0&0
        \end{array}\right] \S^{T}
    \S^{-T} 
    \left[\begin{array}{cccccc}
        d&c&0&0&0&0\\
        c&-d&0&0&0&0\\
        0&0&0&0&b&a\\
        0&0&0&0&a&-b\\
        0&0&b&a&0&1\\
        0&0&a&-b&1&0
      \end{array}\right]\S^{-1}\\
  &=&\underbrace{\S \Y \S^T}_{\T}\underbrace{\S^{-T} \JJ\, \S^{-1}}_{\W} .\\
\end{eqnarray*}
The matrix $\S$ (and hence also $\S^{-1}$) can be chosen to be real
(see \cite[pages 152,153]{Horn_Johnson}), thus $\T,\W$ are real symmetric.
It is readily checked here that the inertia of $\Y$ (and therefore the inertia of $\T$ by
Sylvester's Law of Inertia) and the inertia of $\JJ$ (and correspondingly of $\W$)
are both $(3,3,0)$ in accordance with Corollary \ref{maincor}.
\end{exmpl}

\section*{Conclusions}
\label{sec:conclusions}

Every real square matrix can be expressed as the product of two real
symmetric matrices of the same dimension.  Such a factorisation is
generally non-unique. We have proved that if the given real matrix has
only non-real (i.e. complex) eigenvalues, then the two real symmetric matrices in any
such factorisation must have equal numbers of positive and negative eigenvalues.
More generally, we give bounds on the number of positive and negative eigenvalues
in terms of the number of non-real (i.e. complex) eigenvalues of the original
matrix.


\bibliographystyle{siam}
\bibliography{complexeigs_references}
\end{document}